\documentclass[12pt, english, a4paper]{amsart}

\usepackage{graphicx}

\usepackage{amsmath,amssymb}
\usepackage{bm}
\usepackage{euscript,color}
\usepackage{babel}
\usepackage{amsthm}
\usepackage{amsfonts}
\usepackage{latexsym}
\usepackage{fancyhdr}
\usepackage{enumerate}

\usepackage[left=2.5cm, top=3cm, bottom=3cm, right=2.5cm]{geometry}

\def\mymathfont{\mathbf}

\newcommand{\C}{\mymathfont{C}}
\newcommand{\Z}{\mymathfont{Z}}

\newcommand{\BW}{\mathop{\mathrm{BW}}\nolimits}

\newcommand{\CP}{\mathop{\mymathfont{C}\mathrm{P}}\nolimits}

\newcommand{\Ric}{\mathop{\mathrm{Ric}}\nolimits}

\makeatletter
\renewcommand{\section}{\@startsection%
{section}
{1}
{0mm}
{1.5\bigskipamount}
{0.5\bigskipamount}
{\centering\normalsize\sc}}

\renewcommand{\paragraph}{\@startsection%
{paragraph}
{4}
{0mm}
{\bigskipamount}
{-1.25ex}
{\normalsize\sl}}

\def\provedboxcontents#1{$\square$}

\makeatother

\newtheoremstyle{thm}{}{}{\slshape}{}{\scshape}{.}{0.5em}{}

\newtheoremstyle{def}{}{}{}{}{\scshape}{.}{0.5em}{}

\newtheoremstyle{rmk}{}{}{}{}{\scshape}{.}{0.5em}{}

\newtheoremstyle{claim}{}{}{}{}{\slshape}{.}{0.5em}{}

\theoremstyle{thm}
\newtheorem{newstatement}{newstatement}
\newtheorem{lemma}{Lemma}
\newtheorem{theorem}[newstatement]{Theorem}
\newtheorem{corollary}{Corollary}
\newtheorem{proposition}{Proposition}

\newtheorem*{conjecture*}{Conjecture}

\theoremstyle{def}

\newcommand{\K}{K\"{a}hler}

\theoremstyle{rmk}
\newtheorem{remark}{Remark}
\newtheorem{example}[newstatement]{Example}

\theoremstyle{claim}

\expandafter\let\expandafter\oldproof\csname\string\proof\endcsname
\let\oldendproof\endproof
\renewenvironment{proof}[1][\proofname]{%
  \oldproof[\slshape #1]%
}{\oldendproof}

\let\nothing\varnothing
\let\geq\geqslant
\let\leq\leqslant

\let\epsilon\varepsilon

\renewcommand{\emph}[1]{{\slshape #1}}
\renewcommand{\em}{\sl}

\title[Sasakian immersions into the sphere]{Sasakian immersions into the sphere}

\author{Beniamino Cappelletti--Montano}
\address{Beniamino Cappelletti--Montano, Dipartimento di Matematica e Informatica \\
         U\-ni\-ver\-si\-t\`a di Cagliari, Italy.}
         \email{b.cappellettimontano@unica.it}

\author{Andrea Loi}
\address{Andrea Loi, Dipartimento di Matematica e Informatica \\
         Universit\`a di Cagliari, Italy.}
         \email{loi@unica.it}

\date{}

\thanks{The  authors were   supported by Prin 2015 -- Real and Complex Manifolds: Geometry, Topology and Harmonic Analysis -- Italy and by Fondazione di Sardegna and Regione Autonoma della Sardegna (Project GESTA). All the authors are  members of INdAM-GNSAGA - Gruppo Nazionale per le Strutture Algebriche, Geometriche e le loro Applicazioni.}

\begin{document}
\begin{abstract}
The aim of this paper is to study  Sasakian immersions  of compact   Sasakian manifolds into the odd-dimensional sphere equipped with the standard Sasakian structure. We obtain a complete classification of such  manifolds  in the Einstein and   $\eta$-Einstein cases when the codimension of the immersion is $4$. Moreover, we exhibit infinite families  of compact Sasakian $\eta$--Einstein manifolds which cannot admit a Sasakian immersion into any  odd-dimensional sphere. Finally, we show that, after possibly performing a ${\mathcal {D}}$-homothetic deformation,  a homogeneous Sasakian manifold
can be Sasakian immersed into some odd-dimensional sphere if and only if  $S$ is regular and either $S$ is simply--connected or its fundamental group is finite cyclic.
\end{abstract}
\maketitle
\section{Introduction}

Sasakian manifolds were introduced by the foundational work of Sasaki \cite{sasaki}  in 1960. A \emph{contact metric manifold} is a contact connected manifold $(S,\eta)$ admitting a Riemannian metric $g$ compatible with the contact structure, in the sense that, defined the $(1,1)$-tensor $\phi$ by $d\eta(X,Y)=2g(X,\phi Y)$, the following conditions are fulfilled
\begin{equation}\label{contactmetric}
\phi^2 X = -X + \eta(X)\xi, \quad g(\phi X, \phi Y) = g(X,Y) - \eta(X)\eta(Y),
\end{equation}
where $\xi$ denotes the \emph{Reeb vector field} of the contact structure, that is the unique vector field on $S$ such that
\begin{equation*}
i_{\xi}\eta=1, \ \ i_{\xi}d\eta=0.
\end{equation*}
Moreover, a contact metric manifold is said to be \emph{Sasakian}  if  the following  integrability condition is satisfied
\begin{equation}\label{normal}
[\phi X, \phi Y] = -d\eta(X,Y)\xi,
\end{equation}
for any vector fields $X$ and $Y$ on $S$. It follows from the definition that $S$ must be of odd dimension, say $2n+1$.
Two Sasakian manifolds $(S_{1},\eta_{1},g_{1})$ and $(S_{2},\eta_{2},g_{2})$ are said to be \emph{equivalent} if there exists a contactomorphism $F: S_{1} \longrightarrow S_{2}$ between them which is also an isometry, i.e.
\begin{equation}\label{equivalence1}
F^{\ast} \eta_{2} = \eta_{1}, \quad F^{\ast}g_{2} = g_{1}.
\end{equation}
One can prove that if   \eqref{equivalence1} holds  then $F$ satisfies also
\begin{equation*}
F_{{\ast}_x}\circ \phi_{1} = \phi_{2}\circ F_{{\ast}_x}, \quad F_{{\ast}_x}\xi_{{1}} = \xi_{{2}}
\end{equation*}
for any $x\in S_1$.
An isometric  contactomorphism $F: S \longrightarrow S$ from a Sasakian manifold $(S,\eta, g)$ to itself will be called a \emph{Sasakian transformation} of $(S,\eta, g)$. A Sasakian manifold is  \emph{homogeneous} if it is acted upon  transitively by its group of Sasakian transformations.

Sasakian geometry can be considered as  the odd-dimensional counterpart of K\"{a}hler geometry. In fact in any contact manifold $(S,\eta)$ one can consider the $1$-dimensional foliation defined by the Reeb vector field. Actually one can prove that this foliation is  transversely K\"{a}hler if and only if $S$ is Sasakian. On the other hand a Sasakian manifold can be also characterized as a Riemannian manifold $(S,g)$ whose metric cone $(S\times \mathbb{R}^{+}, r^{2}g + dr^{2})$ is K\"{a}hler. In particular, one can prove that  $(S,g)$ is Sasaki-Einstein if and only if the corresponding Riemannian cone is Calabi-Yau. The classical example of Sasaki-Einstein manifold is given by the odd-dimensional sphere $\mathbb{S}^{2n+1}$ endowed with  the usual Riemannian metric $g_0$ and the contact form induced by the form $x_{1}dy_{1}-y_{1}dx_{1}+\ldots+x_{n+1}dy_{n+1}-y_{n+1}dx_{n+1}$ on $\mathbb{R}^{2n+2}$. This is called the \emph{standard Sasakian structure} of $\mathbb{S}^{2n+1}$. In all the paper, unless otherwise stated, whenever we speak of the $\mathbb{S}^{2n+1}$ as a Sasakian manifold, we are assuming that it is equipped with the standard Sasakian structure $(\eta_0, g_0)$.

 Sasaki-Einstein manifolds attracted the attention of several authors since it was pointed out their relation with string theory and the so-called Maldacena conjecture (see \cite{maldacena}). In this framework, Gauntlett, Martelli, Sparks and Waldram discovered the first known examples of irregular (see section  below for the definition) Sasaki-Einstein metrics on $\mathbb{S}^{2}\times \mathbb{S}^{3}$ (\cite{martelli1}). We mention also the  work of Boyer, Galicki and Koll\'{a}r \cite{BoyerGalickiKollar} on the existence of non-trivial Sasaki-Einstein metrics on the spheres and to the study by Martelli, Sparks and Yau on the relations between the critical points of Einstein-Hilbert action and Sasaki-Einstein manifolds (\cite{martelli2}).

Let us consider now the foliation defined by the Reeb vector field of a Sasakian manifold $S$. Using the theory of Riemannian submersions one can show that the transverse geometry is K\"{a}hler-Einstein if and only if the Ricci tensor of $S$ satisfies the following equality
\begin{equation}\label{etaeinstein}
\textrm{Ric} = \lambda g + \nu \eta\otimes\eta
\end{equation}
for some constants $\lambda$ and $\nu$. Any Sasakian manifold satisfying \eqref{etaeinstein} is said to be \emph{$\eta$-Einstein}. Notice that in any $\eta$-Einstein Sasakian manifold the Einstein constants are related by
\begin{equation}\label{einstein-constants}
\lambda+\nu = 2n
\end{equation}
(see e.g. \cite{BoyerGalickiMatzeu06}).  Another useful property of $\eta$-Einstein Sasakian manifolds is that, contrary to Sasaki-Einstein ones, they are preserved by $\mathcal{D}_a$-homothetic deformations, that is the change of structure tensors of the form
\begin{equation}\label{Dhom}
\phi_{a}:= \phi, \quad \xi_{a}:=\frac{1}{a}\xi, \quad \eta_{a}:=a \eta, \quad g_{a}:=a g + a(a-1)\eta\otimes\eta
\end{equation}
where $a>0$. This transformations were first considered by Tanno in \cite{Tanno68} and then used in several contexts. One proves (see \cite{Blair2010} and \cite{BoyerGalickiMatzeu06}) that if $(\phi,\xi,\eta,g)$ is a Sasakian $\eta$-Einstein structure on $S$ with Einstein constants $(\lambda,\nu)$, then, for any $a>0$, the deformed structure $\left(\phi_{a},\xi_{a},\eta_{a},g_{a}\right)$ is still a Sasakian $\eta$-Einstein structure with Einstein constants given by
\begin{equation}\label{costanti-Einstein}
\lambda_{a} = \frac{\lambda + 2 - 2a}{a}, \quad \nu_{a}  =2n - \frac{\lambda + 2 - 2a}{a}.
\end{equation}

Combining (\ref{etaeinstein}) and  (\ref{costanti-Einstein}) one sees that  the  $\mathcal{D}_a$-homothetic deformation, with
$a=\frac{\lambda+2}{2(1+n)}$,  takes an  $\eta$-Einstein Sasakian strucuture with $\lambda >-2$ into a Sasaki-Einstein one.

Examples of $\eta$-Einstein Sasakian manifolds with  $\lambda >-2$ are provided by the tangent sphere bundle $T_{1}\mathbb{S}^{m}$ of any sphere $\mathbb{S}^{m}$ (see \cite{Tanno79}). Thus, a  suitable  $\mathcal{D}_a$-homothetic deformation give $T_{1}\mathbb{S}^{m}$ the structure of a homogeneous Sasaki-Einstein manifold. In particular, the standard homogeneous Sasaki-Einstein structure on $\mathbb{S}^{2}\times \mathbb{S}^{3} \simeq T_{1}\mathbb{S}^2$ can be obtained in this way.

\bigskip

In this paper we study the Sasakian immersions of  Sasakian manifolds into the odd dimensional sphere. By a Sasakian immersion of a Sasakian manifold $(S_{1},\eta_{1},g_{1})$ into the Sasakian manifold $(S_{2},\eta_{2},g_{2})$ we mean an isometric immersion $\varphi: (S_{1},g_1)\longrightarrow (S_{2},g_2)$ that preserves the Sasakian structures, i.e. such that
\begin{gather}
\varphi^{\ast}g_{2} = g_{1}, \quad \varphi^{\ast} \eta_{2} = \eta_{1}, \label{sasakianimmersion1} \\
\varphi_{{\ast}}\xi_{{1}} = \xi_{{2}}, \quad \varphi_{{\ast}}\circ \phi_{1} = \phi_{2}\circ \varphi_{{\ast}}.\label{sasakianimmersion2}
\end{gather}
This definition was first considered in the early seventies, under different names, by Okumura (\cite{okumura68}), Harada (\cite{harada72}, \cite{harada73}, \cite{harada73-II}), Kon (\cite{kon73}, \cite{kon76}), who mainly studied some geometric conditions ensuring the immersed manifold to be totally geodesic. However, despite the theory of K\"{a}hler immersions, which has widely developed in the last decades due to the fundamental work of Calabi (see \cite{LoiZedda-book} for an updated review of this topic), there are very few results about Sasakian immersions.  Relapsing some conditions in \eqref{sasakianimmersion1}--\eqref{sasakianimmersion2}, we can mention a recent, remarkable result of Ornea and Verbitsky (\cite{OrneaVerb}). Namely they proved that a compact Sasakian manifold admits a CR-embedding (i.e. an embedding, non necessarily isometric, satisfying \eqref{sasakianimmersion2}) into a Sasakian manifold diffeomorphic to a sphere. On the other hand, Takahashi (\cite{tak}) and Tanno (\cite{tan1}) studied codimension one isometric immersions of a Sasakian manifold $S$ in Riemannian manifolds of constant curvature, proving that, under some assumptions, $S$ is of constant curvature $1$.

As far as the knowledge of the authors, no general results concerning Sasakian immersions into the sphere are known. One of the aims of this paper is to start filling this gap.
We start  by the following two classification results (Theorem
\ref{mainteor1} and  Theorem \ref{mainteor3} and the corresponding corollaries),  dealing with Sasaki-Einstein  manifolds, and Sasakian $\eta$--Einstein manifolds in small codimension, respectively.

\begin{theorem}\label{mainteor1}
Let $S$ be a $(2n+1)$-dimensional compact  Sasaki-Einstein manifold. Assume that there exists a Sasakian immersion of $S$ into $\mathbb{S}^{2N+1}$ for some non-negative integer $N$. Then $S$ is Sasaki equivalent to $\mathbb{S}^{2n+1}$.
\end{theorem}

As an  immediate consequence of the theorem one gets:

\begin{corollary}\label{cormainteor1}
The exotic Sasaki--Einstein structures on $\mathbb{S}^{2n+1}$ (\cite{bg})
cannot be induced by a  Sasakian  immersion  into a sphere.
\end{corollary}

A contact metric manifold is said to be \emph{$K$-contact} if the Reeb vector field is Killing. In dimension greater than $3$ this condition is weaker than the Sasakian condition. However, as proven by Boyer and Galicki (\cite{bg1}), and in alternative way by Apostolov, Draghici and Moroianu (\cite{adm}), if the manifold is  compact and  Einstein,  these two notions coincide.  Using this fact and Theorem \ref{mainteor1}, we then  obtain the following:

\begin{corollary}
Let $K$ be a $(2n+1)$-dimensional compact Einstein $K$-contact  manifold. Assume that there exists a contact metric immersion of $K$ into $\mathbb{S}^{2N+1}$ for some non negative integer $N$. Then $K$ is  Sasaki equivalent to $\mathbb{S}^{2n+1}$.
\end{corollary}

In order to state Theorem \ref{mainteor3} we recall the Boothby--Wang construction  (see \cite{Boothby-Wang} and Section \ref{BWS}).
To  any regular and compact Sasakian manifold $(S, \eta, g)$ we can associate a
 compact Hodge manifold $M$, namely a compact  \K\ manifold with integral \K\  form $\omega$ (so $M$ is projective algebraic by Kodaira's theorem) and a principal $\mathbb{S}^1$-bundle $\pi:S\rightarrow M$ with connection $\eta$ such that $\pi^*\omega=ad\eta$, for a  constant $a\neq 0$. The manifold $M$  will be called {\em the \K\ manifold correponding to $S$ through the Boothby--Wang construction}. Notice that if   $(S, \eta_a, g_a)$, $a>0$, is obtained by a regular Sasakian manifold $(S, \eta, g)$ through a $\mathcal {D}_a$-homothetic deformation then  $(S, \eta_a, g_a)$ is still regular and its correspoding  \K\ manifold through the Boothby--Wang construction is the same as that of $(S, \eta, g)$.
Conversely, to  any compact Hodge manifold $M$ one can associate a regular compact Sasakian manifold $(S, \eta, g)$ which is the total space of a principal $\mathbb{S}^1$-bundle  over $M$ and  such that $\pi^*\omega=d\eta$.
Also in this case the  manifold $S$  will be called {\em the  Sasakian manifold correponding to $M$ through the Boothby--Wang construction}.
If $M$ is assumed to be simply-connected then  $S$ is unique up to Sasakian transformations and   will be denoted by $S=\BW(M)$ and  called the  {\em Boothby--Wang manifold} corresponding to $M$ (see Proposition \ref{unique} below for a proof).

\begin{theorem}\label{mainteor3}
Let $S$ be a $(2n+1)$-dimensional  compact $\eta$-Einstein Sasakian manifold. Assume that there exists a Sasakian immersion of $S$ into $\mathbb{S}^{2N+1}$. If $N = n+2$ then  $S$ is Sasaki equivalent to  $\mathbb{S}^{2n+1}$ or to $\BW(Q_n)$, where $Q_n\subset \CP^{n+1}$
is the complex quadric equipped with the restriction of the Fubini--Study form of $\CP^{n+1}$.
\end{theorem}

Theorem \ref{mainteor3} should be compared with  part i) of the main Theorem by Kenmotsu  in \cite{kenmotsu}, where the same conclusion is proved for $N=n+1$ and when  $S$ is assumed to be  complete and not necessarily compact. For general codimension, due to  the corresponding conjecture in the \K\ case (see \cite[Ch. 4]{LoiZedda-book}),  we believe the validity of the following:

\vskip 0.3cm

\noindent
{\bf Conjecture.} {\em If a  compact  $\eta$-Einstein Sasakian manifold can be Sasakian immersed into a sphere then $S$ is Sasakian equivalent to $\BW(M)$ where $M$ is a simply-connected  compact homogeneous Hodge manifold.}

\vskip 0.3cm

The paper contains  two further results (Theorem \ref{mainteor4} and Theorem \ref{mainteor5}).
In Theorem \ref{mainteor4}  (and its Corollary \ref{mainteor4cor})  we  exhibit infinite families of examples of $\eta$-Einstein Sasakian structures on compact manifolds which can not be  induced by the Sasakian structure of the sphere.
In Theorem \ref{mainteor5}  we prove  that the  sphere $\mathbb{S}^{2N+1}$ is, for a suitable $N$, the Sasakian manifold
where all  regular  compact  homogeneous Sasakian manifolds of the form $\BW (M)$  can be Sasakian immersed.

\begin{theorem}\label{mainteor4}
Let $S$ be a compact regular  Sasakian  $\eta$--Einstein manifold of dimension $2n+1$ with Einstein constant $\lambda < 2n$, according to the notation in \eqref{etaeinstein}.  Then $S$ cannot be Sasakian immersed  into any  sphere.
\end{theorem}

\begin{remark}\rm\label{exlambda}
It is worth pointing out that  when  $\lambda\leq -2$ then, by (\ref{costanti-Einstein}), a  $\mathcal {D}_a$-homothetic deformation gives rise to an $\eta$--Einstein structure on $S$ with Einstein constant $\lambda_a\leq -2$.
Thus, by Theorem \ref{mainteor4}, any  $\mathcal {D}_a$-homothetic deformation of an $\eta$-Einstein Sasakian manifold with $\lambda\leq -2$ cannot admit a Sasakian immersion into any sphere.
On the other hand a suitable  $\mathcal {D}_a$-homothetic deformation of an $\eta$-Einstein Sasakian manifold with $-2<\lambda <2n$  gives rise to an  $\eta$-Einstein Sasakian manifold manifold with $\lambda_a\geq 2n$ (and viceversa).
Notice also that the case $\lambda =2n$ corresponds to the Einstein case, treated in Theorem \ref{mainteor1}.
\end{remark}

\begin{corollary}\label{mainteor4cor}
Let $M$ be either   a K3 surface with the Calabi--Yau \K\ form, or the flat complex torus  or a  compact Riemann surface
with the hyperbolic form\footnote{Let
$\Sigma _g$ be
a compact Riemann surface
of genus $g\geq 2$. One can realize
$\Sigma _{g}$ as the  quotient
$D / \Gamma$
of the  unit disk $D\subset{\C}$
where $\Gamma$  is a Fuchsian subgroup $\Gamma\subset SU (1, 1)$.
The \K\ form
$\omega_{hyp}= \frac{i}{2\pi}\frac{dz\wedge d\bar{z}}{(1-|z|^2)^2}$
is invariant by $\Gamma$
so it defines an integral  \K\ form on $\Sigma _g$,
denoted by the
same symbol $\omega_{hyp}$ and called the {\em hyperbolic form}.} and let $S$ be  a regular Sasakian manifold
corresponding to $M$  through the Boothby--Wang construction. Then $S$ which  cannot be Sasakian immersed  into any sphere.
\end{corollary}

\begin{theorem}\label{mainteor5}
Let $S$ be a compact homogeneous Sasakian manifold.
Then, after possibly performing a  $\mathcal {D}_a$-homothetic deformation,  $S$ admits a Sasakian immersion into $\mathbb{S}^{2N+1}$
if and only $S$ is regular and  either $S$ is simply--connected or its fundamental group is finite cyclic.
\end{theorem}

The proof of  Theorem \ref{mainteor1} follows essentially by considering  the induced  \K\ immersion from the Calabi--Yau \K\ cone  $C(S)$ of the Sasakian manifold  $S$ and the \K\ cone of  $\mathbb{S}^{2N+1}$ namely $\mathbb{C}^{N+1}\setminus \{0\}$ and using a result of  Umehara \cite{um} which forces $C(S)$ to be flat.
The proofs of  Theorem \ref{mainteor3} and  Theorem \ref{mainteor4}, for $\lambda\leq -2$, are obtained by considering the induced \K\ immersions into the complex projective space obtained through the Boothby--Wang construction (see Proposition \ref{lemmareg}) and using  some known results on \K\ immersions  due to  D. Hulin \cite{hulinlambda} and K. Tsukada \cite{ts}, respectively.
The case $-2<\lambda < 2n$ in Theorem \ref{mainteor4}  is  treated by the Gauss--Codazzi equations once one cosiders the induced map between the corresponding \K\ cones.

Finally, Theorem \ref{mainteor5}, is based on  a lifting property (Proposition \ref{liftinglemma}) and on the classification of \K\ immersions
of  compact homogeneous \K\ spaces due to the second author, Di Scala and Hishi  \cite{ishi}.

The paper consists in two more  sections.
In Section \ref{BWS} we prove Proposition \ref{lemmareg}, Proposition \ref{unique} and  Proposition \ref{liftinglemma}, while
Section \ref{proofs} is dedicated to the proofs of the main results, Theorems 1-4.

\medskip 
The authors would like to thank Gianluca Bande e Giovanni Placini for a careful reading of the paper.

\section{Boothby--Wang fibrations and Sasakian immersions}\label{BWS}
Let $(S,\eta)$ be a contact manifold and let $\mathcal F$ be the foliation defined by the Reeb vector field. It is well known that $S$ admits an atlas
\begin{equation*}
  \{\left(U_i, \varphi_{i}: U_{i} \longrightarrow \mathbb{R}^{2n+1}=\mathbb{R}\times\mathbb{R}^{2n}\right)\}_{i\in I}
\end{equation*}
such that the change of charts diffeomorphisms $\varphi_{ij}$ locally take the form
\begin{equation*}
\varphi_{ij}(x,y)=\left(g_{ij}(x,y),h_{ij}(y)\right)
\end{equation*}
Each foliated chart is then divided into plaques, the connected components of
\begin{equation*}
\varphi_{i}^{-1}\left(\mathbb{R}\times\left\{y\right\}\right),
\end{equation*}
where $y\in\mathbb{R}^{2n}$, and the change of chart diffeomorphisms preserve this decomposition. Now, $(S,\eta)$ is said to be \emph{regular} if the foliation $\mathcal F$  is regular in the sense of foliation theory. This means that for any $x\in S$ there exists a foliated chart $U$ containing $x$ such that every leaf of $\mathcal F$ intersects at most one plaque of $U$.

The proof of our results  rely on some lemmas on foliation theory. Recall that
a \emph{foliated map} is a differentiable map $f: (X, \mathcal{F}) \longrightarrow (X', \mathcal{F}')$ between foliated manifolds which preserves the foliation structure, i.e. which maps the leaves of $\mathcal{F}$ into leaves of $\mathcal{F}'$, or equivalently, for all $x\in X$,  $f_{{\ast}_{x}}(L(x)) \subset L'(f(x))$, where $L=T(\mathcal F)$ and $L'=T(\mathcal F')$ are integrable distribution of the foliations.
The proof of the following lemma is straighforward.

\begin{lemma}\label{coordinates}
Let $(X, \mathcal{F}$) and $(X', \mathcal{F}')$ be foliated manifolds of dimension $n$ and $n'$, respectively, and $\varphi:X \longrightarrow X'$ be a foliated immersion. Suppose that $\dim(\mathcal{F})=\dim(\mathcal{F'})=p$. Then for each $x\in X$ there are charts $\psi: U \longrightarrow \mathbb{R}^{p}\times\mathbb{R}^{q}$ for $X$ about $x$ and $\psi': U' \longrightarrow \mathbb{R}^{p}\times\mathbb{R}^{q'}$ for $X'$ about $\varphi(x)$ such that
\begin{itemize}
  \item[(i)] $\psi(x)=(0,\ldots,0)\in\mathbb{R}^{n}$
  \item[(ii)] $\psi'(\varphi(x))=(0,\ldots,0)\in\mathbb{R}^{n'}$
  \item[(iii)] $\hat\varphi (x_{1},\ldots,x_{n})=(x_{1},\ldots,x_{n},0,\ldots,0)$, where $\hat\varphi:=\psi'\circ \varphi \circ \psi^{-1}$
  \item[(iv)] $L(x)=\textrm{span}\left\{\frac{\partial}{\partial x_{1}}(x),\ldots,\frac{\partial}{\partial x_{p}}(x)\right\}$, \  $L'(\varphi(x))=\textrm{span}\left\{\frac{\partial}{\partial x_{1}}(\varphi(x)),\ldots,\frac{\partial}{\partial x_{p}}(\varphi(x))\right\}$
\end{itemize}
where  $q=n-p$, $q'=n'-p$.
\end{lemma}

\begin{lemma}\label{lemmareg0}
Let $\mathcal{F}$ and $\mathcal{F}'$ be two (Riemannian) foliations of the same dimension on the  (Riemannian) compact\footnote{It is well known (\cite{palais}) that the quotient space $X/{\mathcal F}$ of a regular foliation on a compact manifold  $X$ has a unique smooth structure which makes the natural  projection $X\rightarrow X/{\mathcal F}$ a submersion.}
manifolds $X$ and $X'$, respectively. If there exists a foliated (isometric) immersion $\varphi: (X,\mathcal{F}) \longrightarrow (X',\mathcal{F}')$ and $\mathcal{F}'$ is regular, then also $\mathcal{F}$ is regular. Furthermore, $\varphi$ descends to an (isometric) immersion $i(\varphi): X/{\mathcal F} \longrightarrow X'/{\mathcal F}'$ between the quotient spaces.
\end{lemma}
\begin{proof}
Assume that $\mathcal{F}$ is not regular. Then there exists a point $x \in X$ and a leaf $L$ of $\mathcal{F}$ such that, for any foliated chart $U$ containing $x$, $L$ intersects more then one plaque in $U$. Let us consider the foliated charts $U$ and $U'$ about $x$ and $\varphi(x)$, respectively, satisfying the properties stated in Lemma \ref{coordinates}. Then there exist at least two plaques, say $P_1=\psi^{-1}\left(\mathbb{R}^{p}\times\left\{\textbf{y}_1\right\}\right)$ and $P_2= \psi^{-1}\left(\mathbb{R}^{p}\times\left\{\textbf{y}_2\right\}\right)$, such that
\begin{equation}\label{intersezione}
L\cap P_1 \neq \nothing, \quad L\cap P_2 \neq \nothing,
\end{equation}
where $\textbf{y}_1,\textbf{y}_2\in\mathbb{R}^{q}$. Note that, for each $i=1,2$, $\varphi(P_i)$ is a plaque of $\mathcal{F}'$ in $U':=\varphi(U)$. Indeed, using Lemma \ref{coordinates}, we have that $\varphi(P_i)=\varphi(\psi^{-1}(\mathbb{R}^{p}\times \left\{\textbf{y}_i\right\}))=\psi'^{-1}(\hat{\varphi}(\mathbb{R}^{p}\times\{\textbf{y}_i\}))=\psi'^{-1}(\mathbb{R}^{p}\times \{(\textbf{y}_i,0,\ldots,0)\})$.  Now, since $\varphi$ is a foliated map,  $L'=\varphi(L)$ is a leaf of $\mathcal{F}'$ and, because of the injectivity of $\varphi$, from \eqref{intersezione} it follows that  $L'\cap \varphi(P_i) \neq \nothing$ for each $i \in \left\{1,2\right\}$. But this contradicts the regularity of $\mathcal{F'}$. For the last part of the statement, note that the construction of $i(\varphi)$, which follows from the leaf-preserving property of $\varphi$, is a well-known fact in foliation theory (see for instance  \cite{KonderakWolak} and \cite{wolak}). It remains to prove that $i(\varphi)$ is an immersion. First, $\varphi$ maps injectively leaves of $\mathcal{F}$ in  leaves of $\mathcal F'$, so that $i(\varphi)$ is injective.   Next, let $M$ and $M'$ denote the leaf spaces of $\mathcal F$ and $\mathcal F'$, respectively, and $\pi : X \longrightarrow M$, $\pi' : X' \longrightarrow M'$ the corresponding global submersions defining the foliations. Fix  Riemannian metrics $g$ on $X$ and  $g'$ on $X'$. Then the tangent spaces at each point of $X$ and $X'$ splits as $T_{x}X=H_{x}\oplus V_{x}$, $T_{x'}X'=H'_{x'}\oplus V'_{x'}$, with $V_{x}=\ker(\pi_{\ast_{x}})=T_{x}\mathcal F$, $V'_{x'}=\ker(\pi'_{\ast_{x'}})=T_{x'}\mathcal F'$, and $H_{x}$, $H'_{x'}$ the corresponding orthogonal complements. Then $\pi_{\ast_{x}}$ and $\pi'_{\ast_{x'}}$ realize isomorphisms respectively between $H_{x}$ and $T_{\pi(x)}M$ and between $H'_{x'}$ and $T'_{\pi'(x')}M'$. Since $i(\varphi)_{\ast}\circ\pi_{\ast}=\pi'_{\ast}\circ \varphi_{\ast}$ and $\varphi$ is an immersion, we conclude that $i(\varphi)_{\ast}$ is injective at any point of $M$. Finally, if  $\mathcal{F}$ and $\mathcal F'$ are Riemannian foliations, $\pi$ and $\pi'$ are Riemannian submersions, i.e. for any $x \in X$ and $x' \in X'$ the maps $\pi_{\ast_{x}}:H_{x}\longrightarrow T_{\pi(x)}M$ and $\pi'_{\ast_{x'}}:H'_{x'}\longrightarrow T_{\pi'(x')}M'$ are isometries. It follows that if  $\varphi$ is an isometric immersion, also $i(\varphi)$ is isometric.
\end{proof}

A special case of Lemma \ref{lemmareg0} is the following result which will be one key ingredient in the proof of our main results.

\begin{proposition}\label{lemmareg}
Let $\varphi: S\rightarrow  S'$ be a Sasakian immersion between two Sasakian manifolds $S$ and $S'$.
Assume  that $S$ and $S'$ are compact and $S'$ is regular.
Then $S$ is regular and there exists a \K\ immersion $i(\varphi):M\rightarrow  M'$ such that $i(\varphi)\circ\pi=\pi'\circ\varphi$,
where $\pi :S\rightarrow M$ and $\pi' :S'\rightarrow M'$  are the Riemannian submersions given by  Boothby-Wang construction.
\end{proposition}
\begin{proof}
Let $\mathcal F$ and ${\mathcal F}'$ denote the $1$-dimensional foliations defined by the Reeb vector fields $\xi$ and $\xi'$  of $S$ and $S'$, respectively. Note that $\mathcal F$ and ${\mathcal F}'$ are Riemannian foliations, since $\xi$ and $\xi'$ are Killing. Being $\varphi$  a Sasakian immersion,  $\varphi_{\ast}\xi=\xi'$ holds. In particular this implies that $\varphi$ is a foliated isometric immersion between the foliated manifolds $(S,\mathcal F)$ and $(S',{\mathcal F}')$. Then, by Lemma \ref{lemmareg0}, $\varphi$ induces an isometric immersion $i(\varphi):M\rightarrow  M'$ such that $i(\varphi)\circ\pi=\pi'\circ\varphi$. It remainss to prove that $i(\varphi)$ is a K\"{a}hler immersion. Notice that, since $\mathcal F$ and $\mathcal F'$ are  transversely K\"{a}hler foliations, the complex structures $J$ and $J'$ of $M$ and $M'$ and the tensors $\phi$ and $\phi'$ of the Sasakian structures of $S$ and $S'$, respectively, satisfy
\begin{equation}\label{commutativity}
J\circ \pi_{\ast} = \pi_{\ast}\circ \phi, \quad J'\circ \pi'_{\ast} = \pi'_{\ast}\circ \phi'.
\end{equation}
Now let $X$ be a tangent vector at a point of $M$. There exists a unique horizontal vector $\bar{X}$ at point of $S$ such that $\pi_{\ast}\bar{X}=X$. Then,  by using \eqref{commutativity} and \eqref{sasakianimmersion2}, we have   $J'(i(\varphi)_{\ast}X)$ \ $=J'(i(\varphi)_{\ast}(\pi_{\ast}\bar{X}))=$ \ $J'(\pi'_{\ast}(\varphi_{\ast}\bar{X}))=$ \ $\pi'_{\ast}(\phi'_{\ast}(\varphi_{\ast}\bar{X}))=$  \ $\pi'_{\ast}(\varphi_{\ast}(\phi \bar{X}))=$ \ $i(\varphi)_{\ast}(\pi_{\ast}(\phi \bar{X}))=i(\varphi)_{\ast}(J(\pi_{\ast}\bar{X}))=i(\varphi)_{\ast}(J X)$. This completes the proof.
\end{proof}

\begin{remark}\rm
The regularity of the Sasakian manifold $S$ in Proposition \ref{lemmareg} was already stated in \cite{harada73-II}. Here we have given a more general and detailed proof.
\end{remark}

\begin{remark}\rm
Since  the exotic Sasaki--Einstein structures on $\mathbb{S}^{2n+1}$ (\cite{bg})
are not regular, Proposition \ref{lemmareg} yields an alternative proof of Corollary \ref{cormainteor1}.
\end{remark}

In the proof of our result we also need to see if we can reverse the construction in Proposition \ref{lemmareg} by lifting a \K\ immersion to a Sasakian one (see Proposition \ref{liftinglemma} below).
So assume that  $M$ is a compact Hodge manifold.
As we have already pointed out in the introduction there exists a compact  regular Sasakian manifold $(S, \eta, g)$ which  has  over it a $\mathbb{S}^1$-bundle $\pi: S\rightarrow M$
with connection $\eta$ such that $d\eta=\pi^*\omega$.
On the other hand,  the integrality of $\omega$ implies the existence of a  holomorphic (ample) line bundle $p: L\rightarrow M$
whose first Chern class is represented by the De Rham cohomology class of the integral \K\  form $\omega$, namely $c_1(L)=[\omega]_{DR}$. Now,  by a  result of  Ornea and  Verbitsky \cite{OrneaVerbCR} there exists  a Hermitian metric $h_*$ on the dual line  bundle $p_*:L^*\rightarrow M$ such that the  bundle $\pi :S\rightarrow M$ is the restriction
of  $p_*:L^*\rightarrow M$ to the subbundle  consisting of unitary vectors of $L^*$, i.e. $S=\{v\in L^* \ | \ h_*(v, v)=1\}$.
The key point point of their proof is that  the cone $C(S)$ of $S$, viewed as a complex manifold, is equal to the set of non-zero vectors of $L^*$.
Hence the $CR$-structure on $(S, \eta, g)$ and  its   $(1, 1)$-tensor $\phi$ is uniquely detemined by the complex structure of $L$.
Actually one can prove (see, e.g. \cite[Section 2]{ze}) that
 $h_*$ is the dual of the Hermitian metric $h$ on $L$ satisfying
$\Ric(h) = \omega$, where  $\Ric(h)$ is the $2$-form on $M$ whose local expression is given by
\[
\Ric(h)= -\frac{i}{2\pi}\partial\bar{\partial}\log h(\sigma(x),\sigma(x))= \omega
\]
for a trivializing holomorphic section $\sigma : U \rightarrow L\setminus \{0\}$ (here $\partial$ and $\bar\partial$ are the standard complex operator associated  to the holomorphic structure of $L$).
Moreover, the contact form $\eta$ can be written in terms of $h_*$ as follows (see, e.g. the first line  of formula  (8) p. 322 in \cite{ze}):
\begin{equation}\label{heta}
\eta =-i\partial h_{*_|S}
\end{equation}

Using these facts we can prove  the following uniqueness result.
\begin{proposition}\label{unique}
Let $M$ be a simply-connected compact Hodge manifold with integral \K\ form $\omega$.
Let $(S_j, \eta_j , g_j)$, $j=1, 2$,  be two Sasakian manifolds
and assume there exist two principal $\mathbb{S}^1$-bundles $\pi_j:  S_i\rightarrow M$
with connection $\eta_j$  such that $d\eta_j=\pi_j^*\omega$, $j=1, 2$.
Then $(S_1, \eta_1 , g_1)$ and $(S_2, \eta_2, g_2)$ are Sasakian equivalent.
\end{proposition}
\begin{proof}
Let $p_{j*}:L_j\rightarrow M$, $j=1, 2$,  be two holomorphic line bundles such that $c_1(L_j)=[\omega]_{DR}$. Since $M$ is simply-connected  these line bundles are holomorphically equivalent and so  there exists a holomorphic diffeomorphism
$\hat F: L_1^*\rightarrow L_2^*$ such that $p_2\circ\hat  F=p_1$.  Let $h_{j*}$, $j=1, 2$, be the Hermitian metric on $L_j^*$
such that $S_j=\{v\in L_j^* \ | \ h_{j*}(v, v)=1\}$. Since the dual Hermitian metric $h_j$ on $L_j$, $j=1, 2$, satisfies
$\Ric (h_j)=\omega$ and $M$ is compact one easily gets  $\hat F^*(h_{2*})=\lambda h_{1*}$ for a positive  constant $\lambda$.
By denoting by  $F$ the
restriction of $\hat F$ to  $S_1$ one  then  gets a diffeomorphism $F: S_1\rightarrow S_2$ such that $\pi_2\circ F=\pi_1$
and, by the above mentioned result of Ornea-Verbitsky,  it preserves the  tensors $\phi_j$ of $S_j$, namely
\begin{equation*}
F_{{\ast}_x}\circ \phi_{1} = \phi_{2}\circ F_{{\ast}_x},
\end{equation*}
for all $x\in S_1$.
Moreover,  by  (\ref{heta}), one gets
\begin{equation*}
F^* \eta_2=F^* ({-i\partial h_2*}_{|S})={-i\partial\hat  F^*h_{2*}}_{|S}={-i\partial\lambda h_{1*}}_{|S}={-i\partial h_{1*}}_{|S}=\eta_1,
\end{equation*}
where we are denoting by the same symbol the $\partial$-operator of $L_{j*}$, $j=1, 2$.
The last two equations imply $F^*g_2=g_1$ and we are done.
\end{proof}

When  $M$ is a simply-connected  compact Hodge manifold we will denote by $\BW (M)$ the  Sasakian manifold, which we call the Boothby--Wang manifold (unique up to Sasakian transformations by the previous Proposition \ref{unique}) such that there exists a principal $\mathbb{S}^1$-bundle $\pi:\BW (M)\rightarrow M$ whose connection form $\eta$ satisfies $\pi^*\omega=d\eta$.

\begin{example}\label{esfond}
When $M=\C P^n$ is the $n$-dimensional complex projective space and $\omega=\omega_{FS}$ is the Fubini-Study \K\ form\footnote{In homogeneous coordinates the Fubini-Study form reads as  $\omega_{FS}=\frac{i}{2\pi}\partial\bar\partial\log (|Z_0|^2+\cdots+ |Z_n|^2)$.}, then $\BW(\C P^n)=\mathbb{S}^{2n+1}$  and  the  Boothby--Wang fibration $\mathbb{S}^{2n+1}\rightarrow \C P^n$ is the Hopf fibration. Notice that in this case the line bundle $L^*$ is the tautological line bundle over $\CP^n$.
\end{example}

\begin{remark}\rm
When $(M, \omega)$ is a  compact but  non simply-connected \K\ manifold  one could  find  an infinite family of non equivalent  regular Sasakian manifolds $(S, \eta)\rightarrow M$ which are the total space of a $\mathbb{S}^{1}$-bundle over $M$ and satisfying $\pi^*\omega=d\eta$.
This happens, for example by taking $M=\Sigma_g$ a compact  Riemann surface of genus $g\geq 2$ with the hyperbolic form $\omega_{hyp}$. Indeed, there exixts an infinite family of non equivalent holomorphic line bundles over $M$ whose first Chern class can be represented by $\omega_{hyp}$ (see, e.g. \cite{gh}) and thus  by Ornea--Verbitsky one gets an infinite family of non-equivalent  regular Sasakian manifolds $S$ which correspond to $M$ through the Boothby-Wang construction. Notice that, by Corollary \ref{mainteor4cor},  none of these Sasakian manifolds can be Sasakian immersed into some  sphere.
 \end{remark}

The following  lifting result is the key ingredient in the proof of Theoreem \ref{mainteor5}.

\begin{proposition}\label{liftinglemma}
Let $M$, $M'$ be simply-connected  compact Hodge manifolds and let  $(\BW(M), \eta, g )$
(resp. $(\BW(M'), \eta', g' )$ be the corresponding Boothby--Wang manifolds.
Given   a \K\ immersion $i:M\rightarrow  M'$ then  there
 exists a Sasakian immersion $\varphi: \BW(M)\rightarrow  \BW(M')$
such that $i\circ\pi=\pi'\circ\varphi$.
\end{proposition}
\begin{proof}
Consider the  pull-back $\mathbb{S}^1$-bundle  $\hat B\stackrel{\hat\pi}{\rightarrow} M$ induced by  $i$ and let  $\psi: \hat B\rightarrow \BW(M')$ be the bundle map (such that $\pi'\circ \psi=i\circ\hat\pi$).
Since $i$ is a  \K\ immersion it follows  that $(\psi^*\eta', \psi^*g')$ is a Sasakian stucture on  $\hat B$ such that
$\hat\pi^*\omega =d (\psi^*\eta')$. As  $M$ is simply-connected, it follows by Propostion \ref{unique} that  there exists a diffeomorphism $F: \BW(M)\rightarrow \hat B$
such that $F^*\psi^*\eta'=\eta$ and $F^*\psi^*g'=g$. Hence $\varphi:=\psi\circ F$ is  the desired lifting.
\end{proof}

\begin{example}\rm\label{exsphere}
It is interesting to construct explicit  Sasakian immersions obtained as a lift of \K\ immersions.
For example if one considers the  Segre embedding (which is a \K\ embedding)
$$i:\C P^1\times \C P^1\rightarrow \C P^3: ([z_0, z_1], [w_0.w_1])\mapsto [z_0w_0, z_0w_1, z_1w_0, z_1w_1]$$
then the map
$$\varphi:T_1\mathbb{S}^3\cong\mathbb{S}^2\times \mathbb{S}^3\rightarrow\mathbb{S}^7: ([z_0, z_1],\xi_0, \xi_1)\mapsto \frac{(\xi_0z_0, \xi_0z_1, \xi_1z_0, \xi_1z_1)}{\sqrt{|z_0|^2+|z_1|^2}},$$
($\mathbb{S}^3=\{(\xi_0, \xi_1)\in\C^2 \ | \ \xi_0^2+\xi_1^2=1\}$ and $\mathbb{S}^2=\CP^1$),
is a Sasakian immersion, where $T_1\mathbb{S}^3\cong\mathbb{S}^2\times\mathbb{S}^3$ is equipped with an
 $\eta$-Einstein Sasakian  structure which can be also  obtained as a   $\mathcal{D}_a$-deformation of the standard
 homogeneous
 Sasaki--Einstein structure on  $T_1\mathbb{S}^3\cong\mathbb{S}^2\times\mathbb{S}^3$
 described in the introduction (cf. \cite{OrneaPiccinni} and Remark \ref{necdef} below).
\end{example}

\section{Proof of the main results}\label{proofs}
\begin{proof}[Proof of Theorem \ref{mainteor1}]
Let $\varphi:S\longrightarrow \mathbb{S}^{2N+1}$ be a Sasakian immersion. Then $\varphi$ induces a \K\ immersion
$\Phi=\varphi \times \textrm{Id}_{\mathbb{R}^{+}} :C(S)\longrightarrow \mathbb{C}^{N+1}\setminus \{0\}$ between the corresponding \K\ cones. As already pointed out in the introduction, it is well known (cf. \cite{bg}) that the \K\ cone $C(S)$ of a Sasaki-Einstein manifold is  Calabi-Yau, i.e.
the  \K\ metric on $C(S)$ is Ricci flat. By a result of Umehara \cite{um} a  Ricci flat metric on a \K\ manifold which admits a \K\ immersion
 into $\mathbb{C}^N$ (equipped with the flat metric) is forced to be flat.
Notice that the curvature tensors $R$ and $\bar{R}$ of the Riemannian manifolds $S$ and $C(S)$, respectively, are related by
\begin{equation*}
\quad \bar{R}\left(X,Y\right)Z=R\left(X,Y\right)Z + g(X,Z)Y - g(Y,Z)X
\end{equation*}
for any $X,Y\in\Gamma(TS)$ (see, for instance, \cite{adm}). Thus, being $C(S)$ flat, $S$ becomes a manifold of constant curvature $1$. By a result of Tanno (\cite{tan2}), locally the Sasakian structure of $S$ is isomorphic to the standard Sasakian structure of the $(2n+1)$-sphere. More precisely, being a complete Riemannian manifold of constant curvature $1$, $S$ is isometric to a quotient $\mathbb{S}^{2n+1} / \Gamma$ of a Euclidean sphere under a finite group of isometries (\cite{wolf}). We claim that $\Gamma$ is the identity group and so $S$ is Sasakian equivalent to $\mathbb{S}^{2n+1}$.
Indeed, let $\pi:\mathbb{S}^{2n+1}\rightarrow \mathbb{S}^{2n+1} / \Gamma$ be the universal covering map. Consider the  Sasakian immersion $f=\varphi\circ\pi:\mathbb{S}^{2n+1}\rightarrow \mathbb{S}^{2N+1}$ and let $i:\mathbb{S}^{2n+1}\hookrightarrow \mathbb{S}^{2N+1}$ be the standard totally geodesic embedding. Then $F=f\times \textrm{Id}_{\mathbb{R}^{+}}$ and $I=i\times \textrm{Id}_{\mathbb{R}^{+}}$ are two \K\ immersions from $\mathbb{C}^{n+1}\setminus \{0\}$ into $\mathbb{C}^{N+1}\setminus \{0\}$ (the latter is the natural inclusion). By the celebrated Calabi's rigidity theorem  (see \cite[Theorem 2]{Cal}) there exists a unitary transformation $U$ of $\mathbb{C}^{N+1}$
such that $U\circ F=I$. Therefore $F$,  and hence $f$,  is forced to be injective. Thus $\pi$ is injective and $\Gamma$ reduces to the identity group, proving our claim.
\end{proof}

\begin{proof}[Proof of Theorem \ref{mainteor3}]
It follows by Proposition \ref{lemmareg} and Example \ref{esfond} that  $S$ is regular and if  $M$ denotes the complex $n$-dimensional compact \K\ manifold given by the Boothby--Wang construction, it admits a \K\ immersion into $\CP^N$, with $N =n+2$. Since $S$ is  compact and   $\eta$-Einstein its base $M$ is a compact  K\"{a}hler-Einstein manifold (cf. \cite{bg}). By a result due to Tsukada \cite{ts} the codimension restriction forces
$M$ to be either the complex quadric  $Q_n\subset \CP^{n+1}$ or $\CP^n$ which are both simply-connected.
Hence the conclusion follows by Proposition \ref{unique}.
\end{proof}

\begin{proof}[Proof of Theorem \ref{mainteor4}]
Let $S$ be an $\eta$-Einstein Sasakian manifolds with Einstein constants $(\lambda,\nu)$ and assume by a contradiction that there exists
a Sasakian immersion  $\varphi:S\longrightarrow \mathbb{S}^{2N+1}$.
We  distinguish two cases:  $-2<\lambda<2n$ and $\lambda\leq -2$. Let us first suppose
 $-2<\lambda<2n$. A straightforward computation shows that the Ricci tensor of the Riemannian cone $C(S)$ of $S$ is given by
\begin{equation}\label{ricci-cone}
\textrm{Ric}_{C(S)}\left(\frac{d}{dr},\cdot\right)=0, \quad \textrm{Ric}_{C(S)}(X,Y)= -\nu\left(g(X,Y)-\eta(X)\eta(Y)\right)
\end{equation}
for any $X,Y\in\Gamma(TS)$. Using \eqref{ricci-cone} one can easily get  a local basis on $C(S)$ with respect to which the Ricci tensor of $C(S)$ is represented by the following matrix
\begin{equation}\label{ricci-matrix}
\textrm{diag}(-\nu,\ldots,-\nu,0,0)
\end{equation}
where the entry $-\nu$ is repeated $2n$ times. Now, our assumption that $-2<\lambda<2n$ together with \eqref{einstein-constants} yield that $0 < \nu < 2+2n$. In particular, in view of \eqref{ricci-matrix}, this implies that the Ricci tensor of the K\"{a}hler cone $C(S)$ is  not negative semidefinite. On the other hand, as in the proof of Theorem \ref{mainteor1}, $\varphi$ induces a \K\ immersion
$\Phi :C(S)\longrightarrow \mathbb{C}^{N+1}\setminus \{0\}$ between the corresponding \K\ cones. Hence by
the Gauss--Codazzi equations (see e.g. \cite[Prop. 9.5, Ch. IX]{kono}) one deduces that the Ricci tensor of $C(S)$ is negative semidefinite, yielding the desired contradiction. Assume now that  $\lambda\leq -2$ and let
$M$ be the K\"{a}hler manifold which corresponds to $S$ through the Boothby-Wang construction.
Using the O'Neill tensors of the theory of Riemannian submersions, one can easily prove that  $M$ is a compact \K--Einstein manifold with scalar curvature $2n(2 + \lambda)\leq 0$.
On the other hand, by Proposition \ref{lemmareg}  and Example \ref{esfond}, the existence of the  Sasakian immersion $\varphi: S\longrightarrow \mathbb{S}^{2N+1}$ would  give rise to a  \K\ immersion
from $M$ into  $\CP^N$. But a result of Hulin
\cite{hulinlambda}  asserts that the scalar curvature of a projectively induced \K-Einstein metric must be strictly positive,
in contrast with the inequality just proved.
\end{proof}

\begin{proof}[Proof of Theorem \ref{mainteor5}]
In order to prove the theorem notice first that  if  $S$ is  a compact homogeneous Sasakian manifold then the compact Hodge manifold $M$ corresponding to $S$ through the Boothby--Wang contruction is  a  compact homogeneous \K\ manifold.  By a  well-known result (see, e.g. \cite[Theorem 8.97]{BE}) $M$
is then  the \K\ product of a flat complex torus and a simply-connected compact homogeneous \K\ manifold and hence, in particular, its fundamental group is either infinite or trivial.

Assume now   that  $S$ admits a Sasakian immersion into a sphere $\mathbb{S}^{2N+1}$, for some $N$.
Then, by Proposition \ref{lemmareg}  and Example \ref{esfond},  $S$ is regular and $M$ admits a \K\ immersion into the  complex projective space $\CP^N$.
Thus   $M$ is forced to be simply-connected since the flat complex torus cannot admit a \K\ immersion into $\CP^N$ (see, e.g. \cite[Theorem 3]{ishi}).
Consider now  the long exact sequence  of homothopy groups associated to the Boothby--Wang fibration
$\pi :S\rightarrow M$:
\begin{equation}\label{exact}
\dots\rightarrow\pi_1(\mathbb{S}^1)\cong\Z\stackrel{\alpha}{\rightarrow}\pi_1(S)\stackrel{\beta}{\rightarrow}\pi_1(M)\rightarrow\pi_0(\mathbb{S}^1)=\{0\}\rightarrow\cdots
\end{equation}
The condition $\pi_1(M)=\{0\}$ implies that the map $\alpha: \Z\rightarrow \pi_1(S)$ is surjective. Thus  $\pi_1(S)$
is isomorphic to  either $\{0\}$, $\Z$ or $\Z_m$ for some integer $m>0$. The possibility $\pi_1(S)=\Z$ is excluded by
the fact that the first Betti number of a compact  Sasakian manifold must be even (\cite{futa}). Then  one implication of theorem follows.

Conversely,   assume that $S$ is a regular  compact homogeneous Sasakian manifold
whose fundamental group is either  trivial or finite cyclic. Let $M$ be the compact homogeneous Hodge manifold corresponding
to $S$ through the Boothby--Wang construction. By  the long exact sequence (\ref{exact}) and the surjectivity of
the map $\beta:\pi_1(S)\rightarrow \pi_1(M)$ we deduce that $\pi_1(M)$ is either trivial or finite cyclic.
Therefore $M$ is forced to be simply-connected since the fundamental group of a torus is not finite.
 Now, any simply-connected  homogeneous compact Hodge manifold admits a \K\ immersion into $\CP^N$, for some $N$ (see Theorem 1 in \cite{LZDGA}). Thus,  by  Proposition \ref{liftinglemma}, we can lift this \K\ immersion to a Sasakian immersion from $\BW (M)$ into $\mathbb{S}^{2N+1}$. Moreover, since  $M$ is simply-connected, then, up to a ${\mathcal D}_a$-homothetic deformation, $S=\BW(M)$ and we are done.
\end{proof}

\begin{remark}\rm\label{necdef}
To understand the necessity of a ${\mathcal D}_a$-homothetic deformation in Theorem \ref{mainteor5}, consider any compact  simply--connected homogeneous Hodge manifold $M$  with an integral  \K\--Einstein form, which has necessarily
strictly positive scalar curvature.  By the  aforementioned Theorem 1 in \cite{LZDGA} $M$ admits a \K\ immersion into
$\CP^N$, for some $N$, and then, by  Proposition \ref{liftinglemma},   its Boothby--Wang Sasakian manifold $\BW(M)$
admits a Sasakian immersion into $\mathbb{S}^{2N+1}$. Theorem \ref{mainteor4}
forces  $\BW(M)$ to be $\eta$-Einstein with $\lambda >2n$. Then, by  a suitable ${\mathcal D}_a$-homothetic deformation  of the Sasakian structure of $\BW (M)$ (cf. Remark \ref{exlambda}) we get a compact and homogeneous  $\eta$-Einstein Sasakian manifold  $S$ with $-2<\lambda_a<2n$ which, by Theorem \ref{mainteor4}, does not admit a Sasakian immersion into any sphere.  (Example \ref{exsphere} above is a particular case of this construction when $M=\CP^1\times \CP^1$).
\end{remark}

We end this paper with an explicit example of compact homogeneous Sasakian manifold with finite cyclic group admitting a Sasakian immersion into the sphere.

\begin{example}
If  $m$ is a positive integer then
$\BW (\C P^n, m\omega_{FS})$  is the Sasakian manifold given by the lens space $\mathbb{S}^{2n+1}\ /\Z_m$
(for $m=1$ one gets Example \ref{esfond} while, for $m=2$, one gets $SO(3)$ with the standard Sasakian structure).
Indeed, one can show (see, e.g. \cite[p. 908]{englisramadanov}) that the  boundary of the disk bundle of the $m$-th power $p:L^{*m}\rightarrow \CP^n$ of the tautological bundle over $\CP^n$ (cf. Example \ref{esfond} above) is diffeomorphic to   $\mathbb{S}^{2n+1}\ /\Z_m$ and
by  Ornea--Verbitsky \cite{OrneaVerbCR} one gets that the restriction
of $p$ to $\mathbb{S}^{2n+1}\ /\Z_m$ is indeed the Boothy--Wang fibration.
Now, since the fundamental group of  $\mathbb{S}^{2n+1}\ /\Z_m$ is $\Z_m$,
Theorem \ref{mainteor5} yields a Sasakian immersion of   $\mathbb{S}^{2n+1}\ /\Z_m$ into $\mathbb{S}^{2N+1}$, for some $N$. More precisely, this immersion is the lift of  the  \K\  immersion,
$V_m:{(\C}P^n, m\omega_{FS})\rightarrow ({\C}P^{\left({n+m\atop n}\right)}, \omega_{FS})$
obtained by a suitable rescaling of the Veronese embedding
(see \cite[Theorem 13]{Cal}) (hence, in this case,    $N=2\left({n+m\atop n}\right)+1$).
\end{example}

\end{document}